\documentclass[12pt,oneside,a4paper]{amsart}  
\usepackage[centertags]{amsmath}
\usepackage{amstext,amssymb,amsopn,amsthm}
\usepackage{nicefrac, esint}
\usepackage{mathrsfs}
\usepackage{dsfont}
\usepackage{bbm}
\usepackage{vmargin}

\usepackage{graphicx}

\theoremstyle{plain}
\newtheorem{thm}{Theorem}

\newtheorem{lem}[thm]{Lemma}
\newtheorem{fact}[thm]{Fact}
\newtheorem{prop}[thm]{Proposition}

\theoremstyle{definition}

\newtheorem{exmp}[thm]{Example}

\newtheorem*{rem*}{Remark}
\newtheorem{rem}{Remark}

\newcommand{\R}{\mathds{R}}

\newcommand{\Z}{\mathds{Z}}
\newcommand{\E}{\mathcal{E}}

\newcommand{\N}{{\mathds{N}}}

\newcommand{\ind}{\mathds{1}}

\newcommand{\cA}{\mathcal{A}}

\newcommand{\cE}{\mathcal{E}}

\newcommand{\suml}{\sum\limits}

\newcommand{\il}{\int\limits}

\newcommand{\iil}{\iint\limits}

\newcommand{\as}[2][r_0]{(K#2,$#1$)}

\DeclareMathOperator{\dist}{dist}
\DeclareMathOperator{\diam}{diam}
\DeclareMathOperator{\supp}{supp}

\begin{document}

\title[Nonlocal Dirichlet forms] {Comparability and regularity estimates for
symmetric nonlocal Dirichlet forms}

\author{Bart{\l}omiej Dyda}
\author{Moritz Kassmann}

\address{Fakult\"{a}t f\"{u}r Mathematik\\Universit\"{a}t Bielefeld\\Postfach 100131\\D-33501 Bielefeld}


\thanks{Both authors have been supported by the German Science Foundation DFG through \\ SFB 701. The first author was additionally supported by MNiSW grant N N201 397137.}



\date{September 30, 2011}

\begin{abstract}
The aim of this work is to study comparability of nonlocal Dirichlet forms. We
provide sufficient conditions on the kernel for local and global comparability.
As an application we prove a-priori estimates in H\"{o}lder spaces for solutions
to 
integrodifferential equations. These solutions are defined with the help of
symmetric nonlocal Dirichlet forms. 
\end{abstract}

\maketitle

\section{Introduction}

If, for every $x\in \R^d$, $A(x)$ is a positive definite matrix which is uniformly
bounded, then for every ball $B \subset \R^d$ and every function $u \in
C^\infty_c(B)$ 
\begin{align}\label{eq:loc_com_two}
 \int\limits_B \langle \nabla u(x), A(x) \nabla u(x) \rangle \, dx
\asymp \int\limits_B |\nabla u(x)|^2 \, dx \,.
\end{align}
This property is crucial for many questions related to partial differential
operators of second order in divergence form and to diffusion processes
generated by local Dirichlet forms. The aim of the present work is to study
similar properties for symmetric nonlocal Dirichlet forms.

Fix $\alpha_0 \in (0,2)$. Let $\mathcal{K}$ denote a family of kernels $k_\alpha:\R^d \times \R^d \to [0,\infty)$ which depend on indices $\alpha \in (\alpha_0,2)$. We consider the corresponding bilinear forms
\[ \iil_{\R^d \R^d} \big(u(y)-u(x)\big)\big(v(y)-v(x)\big) k_\alpha(x,y) \,dy \, dx \,, \qquad u,v \in C^\infty_c(\R^d) \,. \]
We study the question, under which additional assumptions on the kernels $k_\alpha \in \mathcal{K}$
local comparability holds, i.e. for every kernel $k_\alpha \in \mathcal{K}$, small ball $B$ and every
function $u \in C^\infty_c(B)$
\begin{align}\label{eq:assum_aim}\tag{A}
 \iil_{B B} \big(u(y)-u(x)\big)^2 k_\alpha(x,y) \,dy \, dx \asymp
(2-\alpha) \iil_{B B} \frac{\big(u(y)-u(x)\big)^2}{|x-y|^{d+\alpha}} \,dy \, dx \,.
\end{align}
This relation means that the ratio of the two quantities is bounded from below
and above by two uniform positive
constants which do not depend on $k_\alpha \in \mathcal{K}$. 

Note that this problem is interesting and unsettled even if all kernels $k_\alpha$ correspond to one fixed index $\alpha \in (0,2)$. The case $\alpha = 2$ corresponds to (\ref{eq:loc_com_two}).

For $\alpha \in (0,2)$ set $\cA_{d,-\alpha} = \frac{\alpha\Gamma((d+\alpha)/2)}{2^{1-\alpha}\pi^{d/2}\Gamma(1-\alpha/2)}$. Note that $\cA_{d,-\alpha} \asymp\alpha(2-\alpha)$ for all $\alpha \in (0,2)$. Fix $\alpha_0 \in (0,2)$  and $c>0$. A standard example where relation (\ref{eq:assum_aim}) holds true is given by the family $\mathcal{K}=\{k_\alpha | \alpha \in (\alpha_0,2) \} $ where $k_\alpha$ is any kernel satisfying 
\begin{align}\label{eq:assum_k_simple}
 c \cA_{d,-\alpha} |x-y|^{-d-\alpha} \leq k_\alpha(x,y) \leq c^{-1} \cA_{d,-\alpha} |x-y|^{-d-\alpha} 
\end{align}
for almost every $x,y \in \R^d$.

In this work we
give sufficient conditions which  are more general than
(\ref{eq:assum_k_simple}). It is easy to see that (\ref{eq:assum_k_simple}) is
not necessary for (\ref{eq:assum_aim}). Define $\widetilde{\mathcal{K}} =\{\widetilde{k}_\alpha | \alpha \in (\alpha_0,2)\}$ with $\widetilde{k}_\alpha(x,y) = k_\alpha(x,y) (\mathbbm{1}_{\{|x|\leq 0.1 |y| \}} + \mathbbm{1}_{\{|y|\leq 0.1 |x| \}})$ where $k_\alpha$ is any kernel satisfying (\ref{eq:assum_k_simple}), then the kernels $\widetilde{k}_\alpha$ do not satisfy (\ref{eq:assum_k_simple}) but (\ref{eq:assum_aim}) is still satisfied for all $\widetilde{k}_\alpha \in \widetilde{\mathcal{K}}$.

One application of our investigation are local Poincar\'{e}- and
Sobolev\--inequali\-ties, see \cite{MaSh02, Ponce04}. Those
inequalities together with a class of
appropriate
cutoff-functions lead to regularity estimates for symmetric nonlocal
Dirichlet forms. We assume that, for some constant $c>0$, and every $R,\rho \in (0,1)$ there is a
nonnegative function $\tau \in C^\infty(\R^d)$ with $\operatorname{supp}(\tau) =
\overline{B_{R+\rho}}$, $\tau(x) \equiv 1$ on $B_{R}$, and for every $k$ 
\begin{align}\label{eq:assum_cutoff}\tag{B}
\sup\limits_{x \in \R^d} \; \il_{\R^d} \big(\tau(y) - \tau(x)\big)^2   k(x,y) \, dy \leq c \rho^{-\alpha} \,.
\end{align}
Note that for $\alpha = 2$ Assumption
$(\ref{eq:assum_cutoff})$ asks for the existence of a cut-off function $\tau$
with $\sup\limits_{x \in \R^d} |\nabla \tau|^2(x) \leq c \rho^{-2}$. Such
$\tau$ obviously exists.

We are able to establish conditions (A) and
(B) under quite mild assumptions. Let us always assume $k(x,y)=k(y,x)$ which is
not a restriction since our bilinear forms are symmetric. Without mentioning it we always assume that for almost every $x,y\in\R^d$
\begin{equation}\label{A:LU}\tag{K}
 L(x-y) \leq k(x,y) \leq U(x-y) \,,
\end{equation}
for some functions $L,U:\R^d \to [0,\infty)$ satisfying $L(x)=L(-x)$,
$U(x)=U(-x)$ for almost every $x \in \R^d$,
$L\neq 0$ on a set of positive measure, and
\begin{equation}\label{A:U0}\tag{U0}
\int_{\R^d} (|z|^2 \wedge 1) U(z)\,dz \leq C_0 < \infty. 
\end{equation}
 Our main assumptions are
the following:
\begin{itemize}
\item[(U1)]
There exists $C_1 >0$  such that for every $r \in (0,1)$
\begin{equation}\label{U1}
 \int_{B(0,r)} |z|^2 U(z) \,dz \leq C_1 r^{2-\alpha} \,.
\end{equation}
\item[(L1)]
There exist $a>1$ and $C_2$, $C_3$ such that every annulus
$B_{a^{-n+1}}\setminus
B_{a^{-n}}$ ($n=0,1,\ldots$)
contains a ball $B_n$ with radius $C_2 a^{-n}$, such that
\begin{equation}\label{scaling}
 L(z) \geq C_3(2-\alpha) |z|^{-d-\alpha}, \quad z \in B_n.
\end{equation}
\end{itemize}

Then we can prove the following result:
\begin{thm}\label{thm:U1L1implyAB}
Assume the function $k:\R^d \times \R^d \to [0,\infty)$ satisfies (U0), (U1) and
(L1). Then conditions (A) and (B) are satisfied. If the constants $C_0, C_1, C_2,
C_3$ appearing in (U0), (U1), (L1) are independent of $\alpha \in (\alpha_0,2)$ ($\alpha_0>0$),
then so are the
constants in (A) and (B). 
\end{thm}

Let us explain applications of this result. By
$H^{\alpha/2}(\R^d)$ we denote the usual Sobolev space of fractional order
$\alpha/2 \in (0,1)$, see \eqref{eq:Sobolevnorm}. If $\Omega \subset \R^d$ is open and bounded,
then $H^{\alpha/2}(\Omega)$ is the space of measurable functions $f:\Omega \to
\R$ which can be represented as restrictions of $H^{\alpha/2}(\R^d)$
to $\Omega$. $H^{\alpha/2}_{\text{loc}}(\Omega)$ denotes the space of all
measurable functions $f:\Omega \to \R$ such that $\phi f \in
H^{\alpha/2}(\Omega)$ for every $\phi \in C_c^\infty(\Omega)$. 

Conditions (A) and (B) allow us to apply several techniques which were
developed for partial differential operators of second order or local Dirichlet
forms respectively. The following weak Harnack inequality holds true for
supersolutions, see \cite{Kas09, Kas11b}.

\begin{thm} \label{theo:weak_harnack}
Assume (A) and (B) hold true. Let $\alpha_0 \in (0,2)$. There are positive reals
$p_0, c$ such that for every $\alpha \in (\alpha_0,2)$ and $u \in L^\infty(\R^d)
\cap H^{\alpha/2}_{\text{loc}}(B_{1})$ with $u \geq 0$ in $B_1$ satisfying
$\cE(u,\phi) \geq 0$ for every nonnegative $\phi \in C_c^\infty(B_1)$
the following inequality holds:
\[   c \inf\limits_{B_{1/4}} u \geq \big( \fint\limits_{B_{1/2}} u(x)^{p_0} \,
dx \big)^{1/p_0} - c \sup\limits_{x \in B_{1/2}} \int\limits_{\R^d \setminus
B_1} u^-(z) k(x,z) \, dz \,. \] 
The constants $p_0, c$ depend only on $d, \alpha_0$ and on the constants arising in
(A) and (B). 
\end{thm}

Throughout this article the abbreviation '$\sup$' shall denote the
essential supremum and '$\inf$' the essential infimum. It is possible to
combine Theorem \ref{theo:weak_harnack} and Theorem
\ref{thm:U1L1implyAB} in order to obtain regularity estimates. In order to focus
on the main issues we formulate a simple
assumption on $k$ for large values of $|x-y|$. We assume that there is $\gamma
\in (0,\alpha)$ such that  
\begin{align}\label{eq:assum_largejumps}\tag{U2}
\limsup\limits_{R \to \infty} R^\gamma \; \il_{|z|>R} U(z) \, dz \leq 1
\,.
\end{align}

Using conditions (L1), (U0), (U1) and (U2) the following nonlocal version of
DeGiorgi's regularity result can be established \cite{Kas11b}.

\begin{thm}\label{theo:reg_result}
Assume (L1), (U0), (U1) and (U2) hold true. Then there
exist $c>0$, $\beta \in (0,1)$ such that for every $x_0 \in \R^n$,
 $u \in L^\infty(\R^d) \cap H^{\alpha/2}_{\text{loc}}(B_{1}(x_0))$
satisfying $\cE(u,\phi) = 0$ 
for every $\phi \in C_c^\infty(B_1(x_0))$
the following H\"{o}lder estimates holds for almost every $x,y \in
B_{1/2}(x_0)$:
\begin{align}\label{eq:hoelder_final_result} 
|u(x) - u(y)| \leq c \|u\|_\infty |x-y|^\beta \, .
\end{align}
If the constants appearing in (L1), (U0), (U1) and (U2) are independent of $\alpha\in (\alpha_0,2)$,
where $\alpha_0>0$, then so are the constants $c$ and $\beta$.
\end{thm}

Theorem \ref{theo:reg_result} is proved in resp. follows from the works
\cite{Kom95, BaLe02TAMS, ChKu03, CCV10} if one allows the constant $c$ in
(\ref{eq:hoelder_final_result}) to depend on $\alpha \in (0,2)$ with $c(\alpha)
\to +\infty$ for $\alpha \to 2$ and if one imposes a stronger isotropic
condition of the form
\begin{equation*}\label{scaling_iso}
\forall \, z \in B_1 (0):\quad  L(z) \geq C_3 |z|^{-d-\alpha} \,.
\end{equation*}
Techniques which are robust as $\alpha \to 2$ are developed for equations in non-divergence form in \cite{CaSi09}. 

The paper is organized as follows. In Section~\ref{sec:sufficient} we introduce
some notation used in sequel 
and prove Theorem~\ref{thm:U1L1implyAB}, breaking its proof into
three parts, namely Propositions~\ref{prop:U1impliesB}, \ref{prop:upper} and~\ref{prop:homo}.
We also provide an example of a kernel satisfying (A) and (B), but not (L1), see
Example~\ref{ex:cone}. In Section~\ref{sec:regularity} we provide the main
ideas of how to prove Theorem~\ref{theo:reg_result}.


\section{Properties of the bilinear form}\label{sec:sufficient}

In this section we prove Theorem~\ref{thm:U1L1implyAB}. The proof consists of
several propositions and lemmata. At the end of the
section we construct an example of a kernel satisfying (A) and (B), but not
(L1).

Let us fix $\alpha\in (0,2)$ and consider the following quadratic
forms
\begin{align}
\E^k_D(u,u) &= \int_D \!\int_D (u(y)-u(x))^2 k(x,y)\,dx\,dy,\quad u\in L^2(D),\label{eq:Ekd}\\
\E^\alpha_D(u,u) &= \alpha(2-\alpha) \int_D \!\int_D (u(y)-u(x))^2 |x-y|^{-d-\alpha} \,dx\,dy,\quad u\in L^2(D),\label{eq:Ead}
\end{align}
where  $D\subset \R^d$ is some open set.
Furthermore, we define Sobolev norm
\begin{equation}\label{eq:Sobolevnorm}
\|u\|_{H^{\alpha/2}(\R^d)} = \|u\|_{L^2(\R^d)} + \E^\alpha_D(u,u)^{1/2}.
\end{equation}

By $B(x,r)=\{z\in\R^d : |x-z|<r\}$ we denote the Euclidean ball with center $x$ and radius $0<r\leq \infty$,
and we use an abbreviation $B_r = B(0,r)$.
By $S^{d-1}=\{x\in\R^d : |x|=1\}$ we denote the unit sphere.

We define Fourier transform as an isometry of $L^2(\R^d)$ determined by
\[
 \hat{u}(\xi) = (2\pi)^{-d/2} \int_{\R^d} u(x)e^{-i\xi\cdot x}\,dx, \quad u\in L^1(\R^d)\cap L^2(\R^d).
\]

The following lemma contains a useful equivalent formulation of condition (U1).

\begin{lem}\label{lem:U1'}
Condition (U1) is equivalent to the following one:
\begin{itemize}
\item[(U1')]
There exists $C_4 >0$  such that for every $r \in (0,1)$
\begin{equation}\label{U1'}
 \int_{\R^d} \Big(r^2 \wedge |z|^2\Big) U(z) \,dz \leq C_4 r^{2-\alpha}.
\end{equation}
\end{itemize}
If the constants $C_0$ and $C_1$ are independent of $\alpha \in (\alpha_0,2)$, where $\alpha_0>0$, then so is the constant $C_4$,
and vice versa.
\end{lem}

\begin{proof}
Implication (U1') $\implies$ (U1) is obvious, we may take $C_0=C_1 :=C_4$.
We assume now (U1) and we fix $0<r<1$.
We consider $n=0,1,2,\ldots$ such that $2^{n+1}r < 1$ (the set of such $n$'s is empty if $r\geq \frac{1}{2}$).
We have by (U1)
\begin{align*}
\int_{2^nr < |z| \leq 2^{n+1}r} U(z)\,dz
  &\leq
2^{-2n}r^{-2} \int_{2^nr < |z| \leq 2^{n+1}r} |z|^2 U(z)\,dz \\
&\leq 2^{-2n}r^{-2} C_1\;  2^{(n+1)(2-\alpha)} r^{2-\alpha}
 =2^{-n\alpha} 2^{2-\alpha} C_1   r^{-\alpha} .
\end{align*}
After summing over all such $n$ we obtain
\[
\int_{r<|z|\leq 1/2} U(z)\,dz \leq  \frac{2^{2-\alpha} C_1}{1-2^{-\alpha}} r^{-\alpha}.
\]
Finally
\[
\int_{1/2<|z|} U(z)\,dz \leq 
4  \int_{\R^d} (|z|^2 \wedge 1) U(z)\,dz
\leq 4C_0
\leq  4C_0  r^{-\alpha}.
\]
Combining the two inequalities above and (U1) we get (U1') with
$C_4=(\frac{2^{2-\alpha} }{1-2^{-\alpha}}+1)C_1 + 4C_0$.
\end{proof}

In next two propositions we prove the easier part of
Theorem~\ref{thm:U1L1implyAB}.

\begin{prop}\label{prop:U1impliesB}
Condition (U1) implies (B). If the constants $C_0$ and $C_1$
are independent of $\alpha \in (\alpha_0,2)$, where $\alpha_0>0$, then so is the constant in~(B).
\end{prop}

\begin{proof}
Let $\tau\in C^\infty(\R^d)$ be a function satisfying $\supp(\tau)=\overline{B_{R+\rho}}$,
$\tau\equiv 1$ on $B_R$, $0\leq \tau \leq 1$ on $\R^d$ and $|\tau(x)-\tau(y)| \leq 2\rho^{-1}|x-y|$ for all $x$, $y\in \R^d$.
In particular, we have then $|\tau(x)-\tau(y)| \leq (2\rho^{-1}|x-y|) \wedge 1$.
For every $x\in \R^d$ we estimate, using Lemma~\ref{lem:U1'}
\begin{align*}
\int_{\R^d} (\tau(x)-\tau(y))^2 k(x,y)\, dy
 &\leq
\int_{\R^d} \left( (4\rho^{-2}|x-y|^2) \wedge 1 \right) U(x-y)\, dy\\
 &= 4\rho^{-2}
\int_{\R^d} (|z|^2 \wedge \frac{\rho^2}{4}) U(z)\, dz\\
&\leq 2^\alpha C_4 \rho^{-\alpha}.\qedhere
\end{align*}
\end{proof}

In the proof of the next proposition we will need the following fact
\cite{MR1331981}. Its elementary proof may be found in \cite{DiNezzaPalatucciValdinoci},
however one has to go through it and see that the constants do not depend
on $\alpha$, provided one has the factor $\alpha(2-\alpha)$ in front of the
Gagliardo norm \eqref{eq:Ead}, \eqref{eq:Sobolevnorm}.

\begin{fact}\label{fact:extension}
Let $D\subset \R^d$ be a bounded Lipschitz domain, and let $0<\alpha<2$.
Then there exists a constant $c=c(d, D)$ (independent of $\alpha$)
and an extension operator $E:H^{\alpha/2}(D) \to H^{\alpha/2}(\R^d)$ with
norm $\|E\| \leq c$.
\end{fact}

Furthermore, we will need the following Poincar\'e inequality \cite{Ponce04}.
\begin{fact}\label{fact:Poincare}
Let $D\subset \R^d$ be a bounded Lipschitz domain, and let $0<\alpha_0<\alpha<2$.
Then there exists a constant $c=c(d, \alpha_0, D)$ such that
\begin{equation}
\|u - \frac{1}{|D|}\int_D u\, dx \|_{L^2(D)}^2 \leq c \E^\alpha_D(u,u),\quad u\in H^{\alpha/2}(D).
\end{equation}
\end{fact}

Now we are ready to formulate and prove the following comparability result.

\begin{prop}\label{prop:upper}
Assume (U0), (U1) and  let $0<\alpha_0<\alpha<2$.
If $D\subset \R^d$ is a bounded Lipschitz domain, then
there exists a constant $c=c(\alpha_0,d,C_1, C_0, D)$ such that
\begin{equation}\label{Dupper}
 \E^k_D(u,u) \leq c \E^\alpha_D(u,u), \quad u\in H^{\alpha/2}(D).
\end{equation}
The constant $c$ may be chosen such that (\ref{Dupper}) holds for all balls
$D=B_r$ of radius $r<1$, and for all $\alpha\in (\alpha_0,2)$.
\end{prop}
\begin{proof}
By $E$ we denote the extension operator from $H^{\alpha/2}(D)$ to $H^{\alpha/2}(\R^d)$, see Fact~\ref{fact:extension}.
By subtracting  a constant, we may and do assume that $\int_D u\,dx = 0$.
We have by Plancherel formula and Fubini theorem
{\allowdisplaybreaks
\begin{align}
\E^k_{D}(u,u) &\leq
\int_D\!\int_{D-y} (u(y+z)-u(y))^2 U(z)\,dz\,dy \label{Plproof}\\
&\leq
\int_D\!\int_{B(0,\diam D)} (Eu(y+z)-Eu(y))^2 U(z)\,dz\,dy \nonumber\\
&\leq
\int_{B(0,\diam D)}\!\int_{\R^d} (Eu(y+z)-Eu(y))^2 \,dy \, U(z)\,dz \nonumber\\
 &=\int_{\R^d} \left( \int_{B(0,\diam D)} |e^{i\xi\cdot z}-1|^2 U(z)\,dz \right) |\widehat{Eu}(\xi)|^2\,d\xi  \nonumber\\
 &=\int_{\R^d} \left( \int_{B(0,\diam D)}4\sin^2\Big(\frac{\xi\cdot z}{2}\Big)  U(z)\,dz \right) |\widehat{Eu}(\xi)|^2\,d\xi.\label{Plproof2}
\end{align}
}
For  $|\xi| > 2$ we obtain, using (U1')
\begin{equation}\label{largexi} 
 \int 4\sin^2\Big(\frac{\xi\cdot z}{2}\Big) U(z)\,dz \leq
|\xi|^2 \int (|z|^2 \wedge 4|\xi|^{-2}) U(z)\,dz 
\leq 4C_4 |\xi|^\alpha,
\end{equation}
and for $|\xi| \leq 2$
\begin{align*}
 \int 4\sin^2\Big(\frac{\xi\cdot z}{2}\Big) U(z)\,dz &\leq
4 \int \left(\Big|\frac{\xi\cdot z}{2}\Big|^2 \wedge 1\right) U(z)\,dz 
\leq 4C_0.
\end{align*}
Thus
\begin{align}
\E^k_{D}(u,u) &\leq
c' \int_{\R^d} \left(|\xi|^\alpha + 1  \right) |\widehat{Eu}(\xi)|^2\,d\xi\nonumber\\
 &\leq c' \|Eu\|_{H^{\alpha/2}(\R^d)}^2 \leq c \|u\|_{H^{\alpha/2}(D)}^2\nonumber\\
&=c (\E^\alpha_D(u,u) + \|u\|_{L^2(D)}^2)\label{EkdEa}
\end{align}
with  $c=c(d,C_4, C_0, D)$.
Since $\int_D u\,dx=0$, we have by Fact~\ref{fact:Poincare}
\begin{align*}
\E^\alpha_D(u,u) &\geq c(\alpha_0, d, D)  \int_D u^2(x)\,dx.
\end{align*}
and this together with (\ref{EkdEa}) proves (\ref{Dupper}).

By scaling, the last assertion of the Theorem is satisfied with a~constant $c=c(\alpha_0,d, C_4, C_0, B_1)$.
\end{proof}

The proof of the remaining part of Theorem~\ref{thm:U1L1implyAB}, i.e. the
inequality
'$\geq$' in (A) under the assumption (L1), is more difficult. 
We will need the following two technical lemmata.

\begin{lem}\label{lem:Whitney}
Let $0<\alpha_0<\alpha<2$.
We let $\eta\in(0,1)$ and for a ball $B=B(x,r)$ we denote $B^* = B(x,r/\eta)$.
Suppose that for some $c_k$, $r_0>0$ and all $0<r<r_0$ we have
\[
 \E^k_{B^*}(u,u) \geq c_k \E^\alpha_{B}(u,u),
\]
for every function $u$ and every ball $B$ of radius $\eta r$.
Then there exists a constant $c=c(d,\alpha_0,\eta)$,
such that for every ball $B$ of radius $r<r_0$ and every function $u$
\[
 \E^k_{B}(u,u) \geq c c_k \E^\alpha_{B}(u,u).
\]
\end{lem}
\begin{proof}
Fix some  $0<r<r_0$ and a ball $D$ of radius $r$.
We take $\mathcal{B}$ to be a family of balls with the following properties.
\begin{itemize}
\item[(i)]
For some $c=c(d)$ and any $x,y\in D$, if $|x-y|<c\dist(x,D^c)$, then there exists $B\in\mathcal{B}$
such that $x,y\in B$.
\item[(ii)]
For every $B\in\mathcal{B}$, $B^*\subset D$.
\item[(iii)]
Family $\{B^*\}_{B\in\mathcal{B}}$ has the finite overlapping property, that is,
each point of $D$ belongs to at most $M=M(d)$ balls $B^*$, where $B\in\mathcal{B}$.
\end{itemize}
Such a family $\mathcal{B}$ may be constructed by considering Whitney decomposition of $D$
into cubes and then covering each Whitney cube by an appropriate
family of balls.

We have
\begin{align}
\E^k_{D}(u,u) &\geq
 \frac{1}{M^2} \sum_{B\in\mathcal{B}} \int_{B^*} \! \int_{B^*} (u(x)-u(y))^2k(x-y)\,dy\,dx\nonumber \\
&\geq 
\frac{c_k}{M^2} (2-\alpha) \sum_{B\in\mathcal{B}} \int_{B} \! \int_{B} (u(x)-u(y))^2|x-y|^{-d-\alpha}\,dy\,dx \nonumber\\
&\geq 
\frac{c_k}{M^2}  (2-\alpha) \int_{D} \! \int_{|x-y|<c\dist(x,D^c)} (u(x)-u(y))^2|x-y|^{-d-\alpha}\,dy\,dx. \label{new}
\end{align}
By  \cite[Proposition~5 and proof of Theorem~1]{Dyda2}, we may estimate
\begin{align}
\int_{D} \! \int_{|x-y|<c\dist(x,D^c)} & (u(x)-u(y))^2|x-y|^{-d-\alpha}\,dy\,dx \nonumber\\
&\geq
c(\alpha,d) \int_{D} \! \int_{D} (u(x)-u(y))^2|x-y|^{-d-\alpha}\,dy\,dx, \label{old}
\end{align}
with some constant $c(\alpha,d)$. We note that in \cite[proof of Theorem~1]{Dyda2}
the constant depends on the domain in question, but in our case, by scaling, we can take the same
constant independent of the choice of the ball $D$. One may also check that $c(\alpha,d)$ stays bounded when $\alpha\in (\alpha_0,2)$.
By (\ref{new}) and (\ref{old}) the lemma follows.
\end{proof}

\begin{lem}\label{lem:q}
If $q\in L^1(\R^d)$ is a nonnegative function with $\supp q \subset B_\rho$,
then for all $R>0$ and functions $u$
\[
\E^{q\ast q}_{B_R}(u,u) \leq 4 \|q\|_{L^1} \E^{q}_{B_{R+\rho}}(u,u).
\]
\end{lem}
\begin{proof}
We have
\begin{align*}
\E^{q\ast q}_{B_R}(u,u) &=
 \int_{B_R}\!\int_{B_R}\!\int_{B_{R+\rho}} (u(x)-u(y))^2 q(x-z) q(z-y)\,dz\,dy\,dx\\
&\leq
2 \int_{B_R}\!\int_{B_R}\!\int_{B_{R+\rho}} 
  \Big((u(x)-u(z))^2 + (u(z)-u(y))^2\Big)\\
&\qquad\qquad\qquad\qquad\qquad\qquad \times q(x-z) q(z-y)\,dz\,dy\,dx\\
&\leq
4 \int_{B_R}\!\int_{B_{R+\rho}} 
 (u(x)-u(z))^2  q(x-z)\,dz\,dx \int q(y)\,dy\\
&\leq 4 \|q\|_{L^1} \E^{q}_{B_{R+\rho}}(u,u). \qedhere
\end{align*}
\end{proof}

We are now ready to finish the proof of Theorem~\ref{thm:U1L1implyAB}.

\begin{prop}\label{prop:homo}
Assume that $L$  satisfies (L1), and let $0<\alpha_0<\alpha<2$.
Then there exists a constant $c=c(d,\alpha_0,C_2,C_3,a)$, such that for all $0<r<1$
\[
 \E^\alpha_{B_r}(u,u) \leq c \E^L_{B_r}(u,u)
\]
\end{prop}
\begin{proof}
Let
\[
 q_n(z) = (L(z)\wedge C_3(2-\alpha)|z|^{-d-\alpha}) \ind_{B_{a^{-n}} \setminus
B_{a^{-n-1}} }.
\]
Using estimate $(a^\alpha -1)/\alpha \leq (a^2-1)/2$ it is easy to see that
\[
 \|q_n\|_{L^1} \leq c(d,C_3,a) (2-\alpha) a^{n \alpha}.
\]
Let $B_n \subset B_{a^{-n}} \setminus B_{a^{-n-1}}$ be a ball like in the
assumption (L1),
that is, having radius $C_2 a^{-n}$ and such that
\[
L(z) \geq C_3(2-\alpha) |z|^{-d-\alpha} \geq C_3(2-\alpha) a^{(n+1)(d+\alpha)},
\quad z\in B_n \cup -B_n.
\]
We obtain
\begin{align*}
q_n\ast q_n(z) &\geq
  C_3^2 (2-\alpha)^2 a^{2(n+1)(d+\alpha)}\; \ind_{B_n \cup -B_n} \ast \ind_{B_n \cup -B_n}(z)\\
&\geq (2-\alpha)^2 c(d,C_2,C_3,a) a^{nd + 2n\alpha}\; \ind_{B_{C_2a^{-n}} }(z).
\end{align*}
We fix $0<r<1$. Let $n_0$ be the smallest natural number such that
$a^{-n_0} < r/2$. From inequality $L\geq \sum_{n=n_0}^\infty q_n$ and
Lemma~\ref{lem:q} we obtain
\begin{align*}
\E^L_{B_r}(u,u) &\geq
 \sum_{n=n_0}^\infty \E^{q_n}_{B_{r/2+a^{-n}}}(u,u)
\geq  \sum_{n=n_0}^\infty (4\|q_n\|_{L^1})^{-1}\E^{q_n\ast q_n}_{B_{r/2}}(u,u)\\
&\geq 
 c(d,C_2,C_3,a)(2-\alpha) \\
&\quad\times \sum_{n=n_0}^\infty
  \int_{B_{r/2}}\!\int_{B_{r/2}} (u(x)-u(y))^2 a^{n(d+\alpha)} \ind_{B_{C_2 a^{-n}}}(x-y)\,dy\,dx\\
&\geq 
 c'(d,C_2,C_3,a)(2-\alpha)
  \int_{B_{C_2 r/(4a)}}\!\int_{B_{C_2r/(4a)}} \frac{(u(x)-u(y))^2}{|x-y|^{d+\alpha}}\,dy\,dx\\
&=
c'(d,\alpha_0,C_2,C_3,a) \E^\alpha_{B_{C_2r/(4a)}}(u,u).
\end{align*}
The proof is finished by applying Lemma~\ref{lem:Whitney}.
\end{proof}

Let us show that (L1) is not necessary for (A) and (B) to hold. The reason is
that (A) uses only integrated quantities but not pointwise estimates on $k$.
However, Assumption (L1) is weak and useful at the same time.

\begin{exmp}\label{ex:cone}

\begin{figure}
\centering
\includegraphics[width=\textwidth]{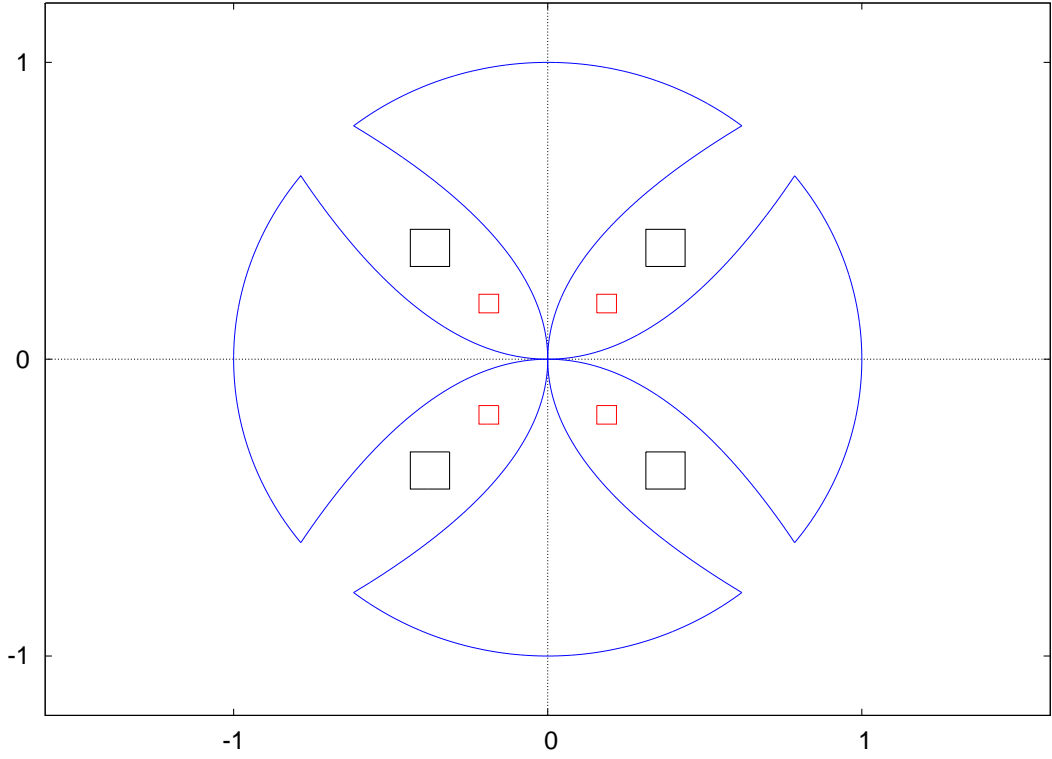}
\caption{Support of the kernel $k$ from Example~\ref{ex:cone} (with $b=1/2$) consists of four thorns.
Also sets $P_0$ and $P_1$ (see (\ref{eq:Pn})) are shown: four larger squares constitute set $P_0$, and four smaller -- set $P_1$. }
\label{fig:thorn}
\end{figure}

Let $b\in(0,1)$ and
\[
 \Gamma = \{ (x_1,x_2)\in \R^2 : |x_2|\geq |x_1|^b \textrm{ or } |x_1|\geq |x_2|^b \}.
\]
We consider the following function
\begin{equation}
k(z) = (2-\alpha) \ind_{\Gamma \cap B_1}(z) |z|^{-2-\beta}, \quad z\in \R^2,
\end{equation}
where $\beta = \alpha-1+1/b$, see Figure~\ref{fig:thorn}.
We will show that for such a function $k$ conditions (A) and (B) are satisfied.

We have, for $0<r<1$
\begin{align}
\int_{B_r} |z|^2 k(z)\,dz &\leq 8(2-\alpha)\int_0^r \int_0^{x^{1/b}} (x^2+y^2)^{-\beta/2} \,dy\,dx\nonumber\\
&\leq 8(2-\alpha)\int_0^r \int_0^{x^{1/b}} x^{-\beta} \,dy\,dx = 
 8 r^{2-\alpha}, \label{qU1}
\end{align}
hence $k$ satisfies (U1) with $C_1 = 8$.
For $n=0,1,2,\ldots$,
we consider the set
\begin{align*}
E_n = & \{ (x_1,x_2): 2^{-n-2}\leq |x_1| \leq 2^{-n-1} \textrm{ and } |x_2| \leq 2^{-2-(n+2)/b} \} \\
& \cup
 \{ (x_1,x_2): 2^{-n-2}\leq |x_2| \leq 2^{-n-1} \textrm{ and } |x_1| \leq 2^{-2-(n+2)/b} \}.
\end{align*}
We have $E_n \subset \Gamma$. Let
\[
 q_n(z) =2^{n(\beta+2)} \ind_{E_n}(z)
\]
and
\begin{equation}\label{eq:Pn}
 P_n = \{ (x_1,x_2): \frac{5}{4} 2^{-n-2} \leq |x_1|, |x_2| \leq \frac{7}{4} 2^{-n-2} \}. 
\end{equation}
If $(x_1,x_2)\in P_n$, then
\begin{align}
q_n \ast q_n(x_1,x_2) &\geq
 \int_{ - 2^{-2-(n+2)/b}}^{2^{-2-(n+2)/b}} \int_{\{|z_2-x_2| \leq 2^{-2-(n+2)/b} \} } 
\!\!\!\!\!\!\!\!\!\!\!\!\!\!\!\!\!\!\!\!\!\!\!\!\!\!\!\!\!\!\!\!\!\!\!\!\!\!
q_n(x_1-z_1,x_2-z_2)  q_n(z_1,z_2)\,dz_2\,dz_1 \nonumber\\
&= 2^{2n(\beta+2)} (2^{-1-(n+2)/b})^2 =2^{-2-4/b} 2^{n(2\alpha+2)}. \label{homog}
\end{align}
We fix $R\in (0,1)$ and take the smallest natural number $n_0$ for which $2^{-n_0}<R/2$.
Since $\|q_n\|_{L^1}= 2^{n\alpha-1-2/b}\leq 2^{n\alpha-2}$, from Lemma~\ref{lem:q} we obtain
\begin{equation}\label{ineq2}
(2-\alpha) \sum_{n=n_0}^\infty 2^{-n\alpha} \E^{q_n\ast q_n}_{B_{R/2}}(u,u) 
 \leq 
   (2-\alpha) \sum_{n=n_0}^\infty  \E^{q_n}_{B_{R}}(u,u) \leq \E^k_{B_{R}}(u,u).
\end{equation}
On the other hand, by (\ref{homog})
\[
(2-\alpha) \sum_{n=n_0}^\infty 2^{-n\alpha} q_n\ast q_n(z)
 \geq (2-\alpha) 2^{-2-4/b} \sum_{n=n_0}^\infty 2^{n(\alpha+2)} \ind_{P_n}(z) =: f(z).
\]
We note that each set $P_n \cap B_{2^{-n}} \setminus  B_{2^{-n-1}}$ contains a~ball
$B_{n+1}$ of radius $c2^{-n-1}$, where $c$ is some universal constant.
Furthermore, on this ball $B_{n+1}$ we have
\[
 f(z) = (2-\alpha) 2^{-2-4/b} 2^{n(\alpha+2)} \geq (2-\alpha) 2^{\alpha-4/b} |z|^{-(\alpha+2)}
\geq (2-\alpha) 2^{-4/b} |z|^{-(\alpha+2)},
\]
provided $n\geq n_0$.
Consequently, $f$ satisfies (L1) with $a=2$ and $n\geq n_0+1$,
or, equivalently, a~rescaled function $x\mapsto f(2^{-n_0-1} x)$ satisfies (L1) with $a=2$.
Since $R/8<2^{-n_0-1}$, we deduce from Proposition~\ref{prop:homo}
\[
 \E^\alpha_{B_{R/8}}(u,u) \leq c \E^f_{B_{R/8}}(u,u) \leq 2^{2+4/b}c \sum_{n=n_0}^\infty 2^{-n\alpha} \E^{q_n\ast q_n}_{B_{R/2}}(u,u).
\]
From this, (\ref{ineq2}) and Lemma~\ref{lem:Whitney} we deduce that
\[
 \E^\alpha_{B_{R}}(u,u) \leq c \E^k_{B_R}(u,u).
\]
The reverse inequality follows from Proposition~\ref{prop:upper}, hence (A) is satisfied. Also (B)
is satisfied by Proposition~\ref{prop:U1impliesB}.
\end{exmp}

\section{Regularity estimates}\label{sec:regularity}

In this section we provide the proof of Theorem \ref{theo:reg_result}. The main
idea of the proof is to extend a result of \cite{Kas09}.

\begin{lem}\label{lem:hh-nonloc-gen}
Assume $x_0 \in \R^d$. For $r\in (0,1)$ and $x \in
B_{r/2}(x_0)$ let $\nu^x_r$ be a measure on $\R^d \setminus
B_r(x_0)$ satisfying 
\begin{align}\label{eq:assum_nu_r}
 \limsup\limits_{j \to \infty} (\eta_{r,j})^{1/j} < 1, \quad \text{ where }
\eta_{r,j} :=
\sup\limits_{x \in B_{r/2}(x_0)} \nu^x_r(\R^d
\setminus B_{2^{j}
r}(x_0) ) < \infty \,.
\end{align}
Assume that for some $c_1 \geq 1$, $p>0$, every $r \in (0,1)$ and every $u \in
L^\infty(\R^d) \cap H^{\alpha/2}_{\text{loc}}(B_{r}(x_0))$
satisfying $\cE(u,\phi) = 0$ 
for every $\phi \in C_c^\infty(B_r(x_0))$ and $u\geq 0$ in $B_{r}(x_0)$, the
weak
Harnack inequality
\begin{align}\label{eq:hoelder-harnack-ass_corr}
\Big( \fint\limits_{B_{r/2}(x_0)} u(x)^p \,dx \Big)^{1/p} \leq c_1
\inf\limits_{x \in B_{r/4}(x_0)} u + c_1 \sup\limits_{x \in
B_{r/2}(x_0)} \il_{\R^d} u^-(z) \nu^x_r (dz)
\,.
\end{align}
holds true. Then there exist $\beta \in (0,1)$, $c > 0$ such that for every
$r \in (0,1)$ and every $u \in
L^\infty(\R^d) \cap H^{\alpha/2}_{\text{loc}}(B_{r}(x_0))$
satisfying $\cE(u,\phi) = 0$ 
for every $\phi \in C_c^\infty(B_r(x_0))$ and every $\rho \in (0,r/2)$ the
following
regularity estimate holds:
\begin{align} \label{eq:hoelder-asser-abstr_corr}
\sup\limits_{x,y \in B_{\rho}(x_0)} |u(x) - u(y)| \leq c \|u\|_\infty
(\rho/r)^{\beta} \,.
\end{align}If $c_1$, $p$ and the limes superior in (\ref{eq:assum_nu_r}) are
independent of $x_0$ or $\alpha$, then so is $c$.
\end{lem}
\pagebreak[3] 

\begin{rem} Instead of (\ref{eq:hoelder-harnack-ass_corr}) one may assume
the strong Harnack inequality
\begin{align}\label{eq:hoelder-harnack-ass_corr_strong}
\sup\limits_{x \in B_{r/4}(x_0)} u \leq c_1
\inf\limits_{x \in B_{r/4}(x_0)} u + c_1 \sup\limits_{x \in
B_{r/2}(x_0)} \il_{\R^d} u^-(z) \nu^x_r (dz) \,.
\end{align}
One only needs to change the constant $c_2$ in the proof of the lemma.
\end{rem}

\begin{proof} The idea is to adopt the methods of \cite{Mos61} to the nonlocal
situation,  see also \cite{Sil06}. Fix $x_0 \in \R^d$. Let $c_1>0$ be the
constant in (\ref{eq:hoelder-harnack-ass_corr}). Let $\theta = 4$. Set
$c_2 = c_1 \theta^{d/p}
2^{\frac{1-d}{p}}$ and $\kappa = (c_2)^{-1}/2$. Let $\beta \in
(0, \ln(\tfrac{2}{2-\kappa}) / \ln(\theta))$ be another constant to be fixed
later. Note that $(1-\tfrac{\kappa}{2}) \leq
\theta^{-\beta}$.

Let $r>0$ and $u \in
L^\infty(\R^d) \cap H^{\alpha/2}_{\text{loc}}(B_{r}(x_0))$
satisfy $\cE(u,\phi) = 0$ 
for every $\phi \in C_c^\infty(B_r(x_0))$. We can
assume $u(x_0) = 0$ which can be obtained by adding a constant if needed. Let
us write $B_r$ instead of $B_r(x_0)$ for $r>0$. 

We will construct an increasing sequence $(m_n)$ and a
decreasing sequence $(M_n)$ satisfying for every $n \in \Z$
\begin{align}\label{eq:hh-nonloc-1}
\begin{split}
&m_n \leq u(x) \leq M_n \quad \text{ for almost all } x \in B_{r \theta^{-n}}
\,, \\
&M_n - m_n \leq K \theta^{-n\beta} \,,
\end{split}
\end{align}
where $K=M_0-m_0\in [0,2\|u\|_\infty]$. Set $M_0=\|u\|_\infty$, $m_0 =
\inf\limits_{\R^d} u(x)$ and $M_{-n}=M_0$, $m_{-n}=m_0$ for every $n \in
\N$. Assume there is $k \in \N$ and there are $M_n, m_n$ such that
(\ref{eq:hh-nonloc-1}) holds for $n \leq k-1$. We need to choose $m_k, M_k$ such
that (\ref{eq:hh-nonloc-1}) holds for $n=k$. 

For $x \in \R^d$ set
\[ v(x) = \Big( u(x) - \frac{M_{k-1}+m_{k-1}}{2} \Big)
\frac{2 \theta^{(k-1)\beta}}{K} \,, \]

The definition of $v$ implies $|v(x)|
\leq 1$ for almost every $x \in B_{r \theta^{-(k-1)}}$ and $\cE(v,\phi) = 0$ 
for every $\phi \in C_c^\infty(B_r)$. 

We now derive a pointwise estimate of the function $v$ on $\R^d
\setminus B_{r
\theta^{-(k-1)}}$. Given $y \in
\R^d$ with $|y-x_0| \geq r \theta^{-(k-1)}$ there is $j \in \N$ such
that
\[ r \theta^{-k+j} \leq |y-x_0| < r \theta^{-k+j+1}   \,. \]
For such $y$ and $j$ we conclude
{\allowdisplaybreaks
\begin{align*}
\frac{K}{2 \theta^{(k-1)\beta}} v(y) &=  \Big( u(y) -
\frac{M_{k-1}+m_{k-1}}{2}
\Big) \\
&\leq \Big(M_{k-j-1} - m_{k-j-1} + m_{k-j-1} - \frac{M_{k-1}+m_{k-1}}{2} \Big)
\\
&\leq \Big(M_{k-j-1} - m_{k-j-1} - \frac{M_{k-1}-m_{k-1}}{2} \Big) \\
& \leq \Big(K
\theta^{-(k-j-1)\beta} - \tfrac{K}{2} \theta^{-(k-1)\beta} \Big) \,,  \\
\text{ i.e. }  v(y) &\leq 2\theta^{j\beta} - 1 \; \leq 2 \Big(\theta
\frac{|y-z|}{r \theta^{-(k-1)}}\Big)^\beta -1  \,.
\end{align*}
}Analogously, 
\begin{align*} 
v(y) \geq 1- 2\theta^{j\beta} \; \geq 1- 2 \Big(\theta
\frac{|y-z|}{r \theta^{-(k-1)}}\Big)^\beta  \,.
\end{align*}
Now there are two cases:

Case 1: $|\{x \in B_{r \theta^{-k}}: v(x) \leq 0 \}| \geq \frac12 |B_{r
\theta^{-k}}|$ 

Case 2: $|\{x \in B_{r \theta^{-k}}: v(x) >  0 \}) \geq \frac12
|B_{r \theta^{-k}}|$ \\

We work out details for Case 1 and comment afterwards on Case 2.
In Case 1 our aim is to show $v(x) \leq 1-\kappa$ for almost every $x \in
B_{r \theta^{-k}}$.
 Because then for
almost every $x \in B_{r \theta^{-k}}$
\begin{align} \label{eq:hh-nonloc-3}
\begin{split}
u(x) &\leq  \frac{M_{k-1}+m_{k-1}}{2}  + \tfrac{(1-\kappa) K}{2}
\theta^{-(k-1)\beta} \\
&= m_{k-1} + \frac{M_{k-1}- m_{k-1}}{2}  + \tfrac{(1-\kappa) K}{2}
\theta^{-(k-1)\beta} \\
&\leq m_{k-1} + \tfrac{K}{2} \theta^{-(k-1)\beta}  + \tfrac{(1-\kappa) K}{2}
\theta^{-(k-1)\beta} = m_{k-1} + (1-\tfrac{\kappa}{2}) K \theta^{-(k-1)\beta} \\
&\leq m_{k-1} + K \theta^{-k\beta} \,.
\end{split}
\end{align}
 In this case we set $m_k=m_{k-1}$
and $M_k= m_{k} + K \theta^{-k\beta}$ and
obtain, using (\ref{eq:hh-nonloc-3}), $m_k \leq u(x) \leq M_k$ for almost every
$x
\in
B_{r \theta^{-k}}$, what needs to be proved.

Let us show $v(x) \leq 1-\kappa$ for almost every $x \in
B_{r \theta^{-k}}$. Consider $w=1-v$. Then $\cE(w,\phi) = 0$ 
for every $\phi \in C_c^\infty(B_{r \theta^{-k+1}})$ and $w
\geq 0$ in $B_{r \theta^{-k+1}}$. We apply (\ref{eq:hoelder-harnack-ass_corr})
and obtain
\begin{align}
\Big( \fint\limits_{B_{\frac12 r \theta^{-k+1}}} w(x)^p \,dx
\Big)^{1/p} \leq c_1 
\inf\limits_{B_{\frac14 r \theta^{-k+1}}} w + c_1 \sup\limits_{x \in
B_{\frac12 r \theta^{-k+1}}} \il_{\R^d} w^-(z) \nu^x_{r \theta^{-k+1}} (dz)
\,.
\end{align}

In the situation of Case 1 we obtain
\begin{align}\label{eq:harnack_applied_corr}
(1/2)^{1/p} &\leq \Big( \fint\limits_{B_{r \theta^{-k}}} w(x)^p \, dx
\Big)^{1/p} \leq (\tfrac{\theta}{2})^{d/p}\Big( \fint\limits_{B_{\frac12 r
\theta^{-k+1}}} w(x)^p \, dx
\Big)^{1/p} \\
&\leq c_1
(\tfrac{\theta}{2})^{d/p} \inf\limits_{B_{r \theta^{-k}}} w + c_1
(\tfrac{\theta}{2})^{d/p} \sup\limits_{x \in
B_{\frac12 r \theta^{-k+1}}}  \il_{\R^d}  w^-(y)
\nu^x_{r \theta^{-(k-1)}}(dy) \,.
\end{align}
For $0<R<S$ let us abbreviate the annulus $B_S(x_0) \setminus B_R(x_0)$ by
$A_{R,S}(x_0)$. Then we obtain  
\begin{align*}
\inf\limits_{B_{r \theta^{-k}}} w &\geq (c_2)^{-1} - \sup\limits_{x \in
B_{\frac12 r \theta^{-k+1}}} \il_{\R^d} w^-(y)
\nu^x_{r \theta^{-(k-1)}}(dy) \\
&\geq (c_2)^{-1} - \suml_{j=1}^\infty \sup\limits_{x \in
B_{\frac12 r \theta^{-k+1}}} 
\il_{\R^d} \mathbbm{1}_{A_{r \theta^{-k+j},r \theta^{-k+j+1}} (x_0)} (1-v(y))^-
\, \nu^x_{r \theta^{-(k-1)}}(dy) \,, \\
&\geq (c_2)^{-1} - \suml_{j=1}^\infty
(2\theta^{j\beta}-2) \eta_{r\theta^{-(k-1)},2(j-1)} \\
&= (c_2)^{-1} - 2 \suml_{j=1}^\infty
(\theta^{j\beta}-1)
\eta_{r\theta^{-(k-1)},2(j-1)} \,.
\end{align*}
Assumption (\ref{eq:assum_nu_r}) guarantees  $\suml_{j=1}^\infty
\theta^{j\beta}
\eta_{r\theta^{-(k-1)},2(j-1)} < \infty$ if $0<\beta<\beta_0$ and $\beta_0$ is sufficiently
small.  Choose $\beta_0$
accordingly. Then there is $l \in \N$ with
\[ \suml_{j=l+1}^\infty (\theta^{j\beta_0}-1)
\eta_{r\theta^{-(k-1)},2(j-1)} \leq \suml_{j=l+1}^\infty \theta^{j\beta_0}
\eta_{r\theta^{-(k-1)},2(j-1)} \leq (8 c_2)^{-1} \,. \]
Given $l$ we choose $\beta \in (0, \beta_0)$ sufficiently small such that 
\[ \suml_{j=1}^{l} (\theta^{j\beta}-1)
\eta_{r\theta^{-(k-1)},2(j-1)} \leq (8 c_2)^{-1} \,. \]
Thus $w
\geq \kappa$ on $B_{r \theta^{-k}}$ or , equivalently, $v \leq
1-\kappa$ on $B_{r \theta^{-k}}$.

In Case 2 our aim is to show $v(x) \geq -1+\kappa$. This time, set $w=1+v$.
Following
the strategy above one sets $M_k=M_{k-1}$ and $m_k= M_{k} - K
\theta^{-k\beta}$ leading to the desired result.

Let us show how (\ref{eq:hh-nonloc-1}) proves the assertion of the
lemma. Let $\rho \in (0,r/2)$. Choose $m\in \N_0$ with $r \theta^{-(m+1)} \leq
\rho < r
\theta^{-m}$. Then condition (\ref{eq:hh-nonloc-1})
implies 
\[  \sup\limits_{x,y \in B_{\rho}(x_0)} |u(x) - u(y)| \leq K \theta^{-m
\beta} =(r \theta^{-m-1})^\beta r^{-\beta} K
\theta^\beta \leq K \theta^{\beta} \Big( \frac{\rho}{r} \Big)^\beta  \,.
\]
The assertion of the lemma follows and the proof is complete.
\end{proof}


Let us explain the proof of our main application. 

\begin{proof}[Proof of Theorem \ref{theo:reg_result}]
The proof of Theorem \ref{theo:reg_result} follows from Lemma
\ref{lem:hh-nonloc-gen} if we can
show that, for every $r \in (0,1)$ and every $u \in
L^\infty(\R^d) \cap H^{\alpha/2}_{\text{loc}}(B_{r}(x_0))$
satisfying $\cE(u,\phi) = 0$ 
for every $\phi \in C_c^\infty(B_r(x_0))$ and $u\geq 0$ in $B_{r}(x_0)$, the
weak
Harnack inequality (\ref{eq:hoelder-harnack-ass_corr}) holds true with
$(\nu^x_r)_{x\in B_{r/2}(x_0)}$ satisfying (\ref{eq:assum_nu_r}).

Fix $x_0 \in \R^d$. Note that none of the constants below will depend on
$x_0$. For $r\in (0,1)$ and $x \in B_{r/2}(x_0)$ define a measure
$\nu^x_r$ on
$\R^d \setminus B_r(x_0)$ by 
\[ \nu^x_r(A) = \int\limits_A U(y-x) \, dy \; \big( \int\limits_{\R^d
\setminus B_r(x_0)} U(y-x_0) \, dy \big)^{-1}  \]
for every Borel set $A \subset \R^d \setminus B_r(x_0)$. Assumption (U2)
implies that there are $c_1 >0$ and $R_0>1$ such that for every $R > R_0$, $r\in
(0,1)$ and $x \in B_{r/2}(x_0)$  
\begin{align}\label{eq:est_U}
 \int\limits_{\R^d\setminus B_{R}(x_0)} U(z-x) \, dz \leq c_1 R^{-\gamma} 
\end{align}
Because of Assumptions (\ref{A:LU}) and (L1) there is $c_2 >0$ with  
\begin{align}\label{eq:est_Ukomp}
\Big(\int\limits_{\R^d\setminus B_{r}} U(z) \, dz\Big)^{-1}  &\leq 
\Big(\int\limits_{\R^d\setminus B_{r}} L(z) \, dz\Big)^{-1} \leq c_2 C_3
r^{\alpha} \,.
\end{align}
Estimates (\ref{eq:est_U}) and (\ref{eq:est_Ukomp}) imply:
\begin{align*}
& \exists c_3 \geq 1 \; \forall r \in (0,1) \; \exists j_0 \geq 1 \; \forall j
\geq j_0 \; \forall x \in B_{r/2}(x_0): \\
& \qquad 
\nu^x_r\big(\R^d \setminus B_{2^{j} r}(x_0)\big) \leq c_3
(2^{j}r)^{-\gamma}/r^{-\alpha} \leq c_3 2^{-\gamma j} \,.
\end{align*}
Recall that we assume $\gamma < \alpha$ in (U2). Condition
(\ref{eq:assum_nu_r}) now
holds
true because of $2^{-\gamma} < 1$ and $c_3^{1/j} \to 1$ for $j \to \infty$.

Let $r \in (0,1)$ and $u \in
L^\infty(\R^d) \cap H^{\alpha/2}_{\text{loc}}(B_{r}(x_0))$
satisfy $\cE(u,\phi) = 0$ 
for every $\phi \in C_c^\infty(B_r(x_0))$ and $u\geq 0$ in $B_{r}(x_0)$. Then
Theorem \ref{theo:weak_harnack} implies
\[   c_4 \inf\limits_{B_{r/4}(x_0)} u \geq \big( \fint\limits_{B_{r/2}(x_0)}
u(x)^{p_0} \,
dx \big)^{1/p_0} - c_4 \sup\limits_{x \in B_{r/2}(x_0)} r^{\alpha}
\int\limits_{\R^d \setminus
B_r(x_0)} u^-(z) k(x,z) \, dz  \] 
with some appropriate constant $c_4>0$. Here we replaced the radius One by
some arbitrary radii $r
\in (0,1)$. This is possible since $(L1), (U1)$ and $(U2)$ allow for scaling.
Finally,
we note that, with some $c_5 >0$ 
\begin{align*}
&\sup\limits_{x \in B_{r/2}(x_0)} r^{\alpha} \int\limits_{\R^d
\setminus
B_r(x_0)} u^-(z) k(x,z) \, dz \\
 &\leq c_5 \big( \int\limits_{\R^d
\setminus B_r(x_0)} U(y-x_0) \, dy \big)^{-1} \sup\limits_{x \in
B_{r/2}(x_0)} \int\limits_{\R^d
\setminus
B_r(x_0)} u^-(z) U(z-x) \, dz \,,
\end{align*}
where we used Assumption (\ref{A:LU}) and the estimate 
\begin{align}
 \int\limits_{\R^d\setminus B_{r}} U(z) \, dz  &\leq r^{-2}
\int\limits_{\R^d \setminus B_{r}} (r^2 \wedge
|z|^2) U(z) \, dz \leq C_4 r^{-\alpha} \,,
\end{align}
which follows from Lemma \ref{lem:U1'}. Condition
(\ref{eq:hoelder-harnack-ass_corr}) now follows. The proof is complete. 
\end{proof}

\section{Appendix}\label{sec:appendix}

In this appendix we provide a global comparability result, i.e. we study
comparability in the whole of $\R^d$.

\begin{prop}\label{prop:upperRd}
If (U0) and (U1) hold, then there exists a constant $c=c(\alpha,d, C_1, C_0)$ such that
\begin{equation}\label{Rdupper2Rd}
 \E^k_{\R^d}(u,u) \leq c (\E_{\R^d}^\alpha(u,u) + \|u\|_{L^2(\R^d)}^2), \quad u\in L^2(\R^d).
\end{equation}
Furthermore, if (\ref{U1}) is satisfied for \emph{all} $r>0$, then
\begin{equation}\label{RdupperRd}
 \E^k_{\R^d}(u,u) \leq c \E_{\R^d}^\alpha(u,u), \quad u\in L^2(\R^d).
\end{equation}
If the constants $C_0$ and $C_1$ in (U0) and (U1) are independent of $\alpha \in (\alpha_0,2)$, where $\alpha_0>0$,
then so are the constants in \eqref{Rdupper2Rd} and \eqref{RdupperRd}.
\end{prop}
\begin{proof}
By $E$ we denote the identity operator from $H^{\alpha/2}(\R^d)$ to itself.
One easily checks that the proof of Proposition~\ref{prop:upper} from (\ref{Plproof}) until
(\ref{EkdEa}) works also in the present case of $D=\R^d$.
Hence (\ref{Rdupper2Rd}) follows.

To prove (\ref{RdupperRd}) we first observe that when (\ref{U1}) holds for all $r>0$, then
we may also get inequality (\ref{U1'}) in Lemma~\ref{lem:U1'} for \emph{all}
$r>0$.
Consequently, (\ref{largexi}) holds for \emph{all} $r>0$, we plug it into (\ref{Plproof2})
and we are done.
\end{proof}

We consider the following condition.
\begin{itemize}
\item[\as{2}]
 There exists $c_0 >0$ such that for all
$h \in S^{d-1}$ and all $0<r< r_0$
\begin{equation}\label{A0}
 \int_{\R^d} r^2 \sin^2\Big(\frac{h \cdot z}{r}\Big) L(z) \,dz \geq c_0
r^{2-\alpha}.
\end{equation}
\end{itemize}

Clearly (L1) implies \as{2} for $r_0=1$, and if $C_1$ is independent of $\alpha \in (\alpha_0,2)$, where $\alpha_0>0$,
then so is $c_0$.
Condition \as{2} is also satisfied if for all $h \in S^{d-1}$ and all $0<r< r_0$
\begin{equation}\label{A2}
 \int_{B(0,r)} |h \cdot z|^2 L(z) \,dz \geq c_2 r^{2-\alpha}.
\end{equation}
We note that (\ref{comp0}) under condition (\ref{A2}) has been proved in \cite{AbelsHusseini}
by Abels and Husseini.
The following theorem extends their result by giving a \emph{characterisation}
of functions $L$
admitting comparability (\ref{comp0}).
We stress that $r_0=\infty$ is allowed, and in such a case we put $\frac{1}{r_0^\alpha} = 0$.

\begin{thm}\label{thm:Rd}
Let $0<r_0\leq \infty$.
If \as{2} holds, then
\begin{equation}\label{comp0}
 \E^\alpha_{\R^d}(u,u) \leq \frac{1}{c_0} \E^k_{\R^d}(u,u) + \frac{2^\alpha}{r_0^\alpha} \|u\|_{L^2}^2\,,
\quad u\in C_c^1(\R^d).
\end{equation}
Conversely, if for some $c<\infty$
\begin{equation}\label{comp0conv}
 \E^\alpha_{\R^d}(u,u) \leq c \iint (u(x)-u(y))^2 L(x-y)\,dy\,dx +
\frac{2^\alpha}{r_0^\alpha} \|u\|_{L^2}^2\,,\quad
u\in \mathcal{S}(\R^d),
\end{equation}
then \as{2} holds.
\end{thm}
\begin{proof}
We change the variable $x$ to $y+z$ and use Plancherel formula. Recalling that $(u(\cdot+z))^\wedge(\xi) = e^{i\xi\cdot z}\hat{u}(\xi)$
we obtain
\begin{align}
\E^k_{\R^d}(u,u) &\geq
\iint (u(x)-u(y))^2 L(x-y)\,dy\,dx \nonumber\\
 &=\int \left( \int |e^{i\xi\cdot z}-1|^2 L(z)\,dz \right)
|\hat{u}(\xi)|^2\,d\xi  \nonumber\\
 &=\int \left( \int 4\sin^2\Big(\frac{\xi\cdot z}{2}\Big) L(z)\,dz \right)
|\hat{u}(\xi)|^2\,d\xi. \label{Ek}
\end{align}
If \as{2} holds, then for all $|\xi| > 2/ r_0$
\begin{align*}
 \int 4\sin^2\Big(\frac{\xi\cdot z}{2}\Big) L(z)\,dz &\geq
\frac{4c_0}{2^\alpha} |\xi|^\alpha \geq c_0|\xi|^\alpha.
\end{align*}
For $|\xi| \leq 2/ r_0$ we have $|\xi|^\alpha \leq (2/r_0)^\alpha$.
Inequality (\ref{comp0})  follows from 
\begin{equation}\label{Ealpha}
\frac{\cA_{d,-\alpha}}{2\alpha(2-\alpha)} \E^\alpha_{\R^d}(u,u) = \int_{\R^d}|\xi|^\alpha |\hat{u}(\xi)|^2\,d\xi.
\end{equation}
Now we prove the converse. Assume (\ref{comp0conv}).
By (\ref{Ek}), the right hand side of (\ref{comp0conv}) equals
\[
 \int \left( c\int 4\sin^2\Big(\frac{\xi\cdot z}{2}\Big) L(z)\,dz  +
\frac{2^\alpha}{r_0^\alpha}\right) |\hat{u}(\xi)|^2\,d\xi,
\]
hence by (\ref{Ealpha}) and (\ref{comp0conv}) we obtain that
\begin{equation}\label{closetoA0}
c\int 4\sin^2\Big(\frac{\xi\cdot z}{2}\Big) L(z)\,dz  +
\frac{2^\alpha}{r_0^\alpha} \geq |\xi|^\alpha, 
  \quad \textrm{for a.e. $\xi\in\R^d$.}
\end{equation}
By continuity of the function
\[
 \R^d\setminus\{0\} \ni \xi \mapsto \int 4\sin^2\Big(\frac{\xi\cdot z}{2}\Big)
L(z)\,dz,
\]
(\ref{closetoA0}) holds for all $\xi\in\R^d$. For $|\xi|\geq 2^{1+1/\alpha}r_0^{-1}$
we have by (\ref{closetoA0})
\[
c\int 4\sin^2\Big(\frac{\xi\cdot z}{2}\Big) L(z)\,dz   \geq
\frac{|\xi|^\alpha}{2},
\]
and hence \as[2^{-1/\alpha}r_0]{2}{}  holds with $c_0=2^{\alpha-3}c^{-1}$.
Since
\[
 \sin^2\Big(\frac{h \cdot z}{2r}\Big) \geq \frac{1}{4}  \sin^2\Big(\frac{h \cdot z}{r}\Big),
\]
also \as{2} holds with \emph{some} constant $c_0$.
\end{proof}


\def\cprime{$'$}

\end{document}